\documentclass[11 pt]{amsart}
\newtheorem{theorem}{Theorem}

\theoremstyle{definition}

\theoremstyle{proposition}
\newtheorem{proposition}{Proposition}
\theoremstyle{remark}

\theoremstyle{corollary}

\begin{document}

\title[Fixed points of circle actions]{Fixed points of circle actions on spaces with rational cohomology of $S^n \vee S^{2n} \vee S^{3n}$ or $ P^2(n) \vee S^{3n}$}

\author{Mahender Singh}

\address{School of Mathematics, Harish-Chandra Research Institute, Chhatnag Road, Jhunsi, Allahabad 211019, INDIA}

\email{mahen51@gmail.com}



\begin{abstract}
Let $X$ be a finitistic space with its rational cohomology isomorphic to that of the wedge sum $P^2(n)\vee S^{3n} $ or $S^{n} \vee S^{2n}\vee S^{3n}$. We study continuous $\mathbb{S}^1$ actions on $X$ and determine the possible fixed point sets up to rational cohomology depending on whether or not $X$ is totally non-homologous to zero in $X_{\mathbb{S}^1}$ in the Borel fibration $X\hookrightarrow X_{\mathbb{S}^1} \longrightarrow B_{\mathbb{S}^1}$. We also give examples realizing the possible cases.
\end{abstract}

\maketitle

\section{Introduction}
This paper is concerned with the fixed point sets of continuous $\mathbb{S}^1$ actions on certain types of spaces introduced by Toda \cite{T}. Toda studied the cohomology ring of a space $X$ having only non trivial cohomology groups $H^{in}(X; \mathbb{Z})=\mathbb{Z}$ for $i=$ 0, 1, 2 and 3, where $n$ is a fixed positive integer. If $u_i \in  H^{in}(X; \mathbb{Z})$ is a generator for $i=$ 1, 2 and 3, then the ring structure of $H^*(X; \mathbb{Z})$ is completely determined by the integers $a$ and $b$ such that
 $$u^2_1=au_2 ~\textrm{and}~ u_1u_2=bu_3.$$
Such a space is said to be of type $(a,b)$. Note that if $n$ is odd, then $u_1^2=0$ and hence $a = 0$. We shall write $X\simeq_{\mathbb{Q}} Y$ if there is an abstract isomorphism of graded rings $H^*(X; \mathbb{Q}) \stackrel{\cong}{\rightarrow} H^*(Y; \mathbb{Q})$. Similarly, by $X\simeq_{\mathbb{Q}} P^h(n)$ we mean that $H^*(X; \mathbb{Q}) \cong \mathbb{Q}[z]/z^{h+1}$, where $z$ is a homogeneous element of degree $n$.

Recall that, given spaces $X_i$ with chosen base points $x_i \in X_i$ for $i=$ 1, 2, ..., $n$, their wedge sum $\vee_{i=1}^n X_i$ is the quotient of the disjoint union $\sqcup_{i=1}^n X_i$ obtained by identifying the points $x_1, x_2, ..., x_n$ to a single point.\\
It is clear that if $b \neq 0$, then
$$X\simeq_{\mathbb{Q}} S^n \times S^{2n}~\textrm{for} ~ a = 0$$
or
$$X\simeq_{\mathbb{Q}} P^3(n)~ \textrm{for}~ a \neq 0.$$
And, if $b = 0$, then
$$X\simeq_{\mathbb{Q}} S^n \vee S^{2n} \vee S^{3n} ~ \textrm{for}~a = 0$$
or
$$X\simeq_{\mathbb{Q}} P^2(n) \vee S^{3n}~ \textrm{for} ~ a \neq 0.$$

\indent Let the group $G= \mathbb{S}^1$ act continuously on a space $X$ of type $(a,b)$. This gives a fibration $X\hookrightarrow X_G \longrightarrow B_G$ called Borel fibration (see \cite{Bo}, Chapter IV). Here $X_G=(X \times E_G) /G$ is the orbit space of the diagonal action on $X \times E_G$ and $B_G$ is the base space of the universal principal $G$-bundle $G \hookrightarrow E_G \longrightarrow B_G$. We say that $X$ is totally non-homologous to zero in $X_G$ if the inclusion of a typical fiber $X\hookrightarrow X_G$ induces a surjection in the cohomology $H^*(X_G; \mathbb{Q}) \longrightarrow H^*(X; \mathbb{Q})$. This condition is equivalent to a nice relation between the ranks of the cohomology of the space and the fixed point set (Proposition 3).

The cohomological nature of the fixed point sets of actions of cyclic groups $\mathbb{Z}_p$ on spaces of type $(a,b)$ has been studied in detail. For any prime $p$, when $b \not\equiv 0$ mod $p$, Bredon (\cite{B1}, \cite{B2}) and Su (\cite{S1}, \cite{S2}) have investigated the fixed point sets of $\mathbb{Z}_p$ actions. For odd primes $p$, when $b \equiv 0$ mod $p$, Dotzel and Singh (\cite{D1}, \cite{D2}) have determined the cohomological nature of the fixed point sets of $\mathbb{Z}_p$ actions. Singh \cite{Si} has investigated the fixed point sets for $p = 2$ and $b \equiv 0$ mod 2.\\
For $b \neq 0$, the cohomological nature of the fixed point sets of $\mathbb{S}^1$ actions has been studied in detail by Bredon (\cite{B1}, \cite{B2}). The aim of this paper is to study $\mathbb{S}^1$ actions for the case $b= 0$ and determine the possible fixed point sets up to rational cohomology.

Before we state our results, recall that, a paracompact Hausdorff space is said to be finitistic if every open covering has a finite dimensional open refinement, where the dimension of a covering is one less than the maximum number of members of the covering which intersect non-trivially. The concept of finitistic spaces was introduced by Swan \cite {Swan} for his work in fixed point theory. It is a large class of spaces including all compact Hausdorff spaces and all paracompact spaces of finite covering dimension. We prove the following results:\\

\begin{theorem} Let $G = \mathbb{S}^1$ act on a finitistic space $X$ of type $(a,0)$ with fixed point set $F$. Suppose $X$ is totally non-homologous to zero in $X_G$, then $F$ has at most four components satisfying the following:
\begin{enumerate}
 \item If $F$ has four components, then each is acyclic and $n$ is even.
\item If $F$ has three components, then $n$ is even and
\begin{center}
 $F \simeq_{\mathbb{Q}} S^r \sqcup \{point_1 \} \sqcup \{ point_2 \}$ for some even integer $2\leq r\leq 3n$.
\end{center}
\item If $F$ has two components, then either
\begin{center}
 $F \simeq_{\mathbb{Q}} S^r \sqcup S^s $ or $(S^r \vee S^s) \sqcup \{point \}$ for some integers $1 \leq r, s \leq 3n$
\end{center}
or
\begin{center}
 $F \simeq_{\mathbb{Q}} P^2(r) \sqcup \{ point \}$ for some even integer $2 \leq r \leq n$. 
\end{center}
\item If $F$ has one component, then either
\begin{center}
 $F\simeq_{\mathbb{Q}} S^r \vee S^s \vee S^t$ for some interers $1 \leq r, s, t \leq 3n$
\end{center}
or
\begin{center}
 $F\simeq_{\mathbb{Q}} S^s \vee P^2(r)$ for some integers $1 \leq r \leq n$ and $1 \leq s \leq 3n$.
\end{center}
\end{enumerate}
Moreover, if $n$ is even, then $X$ is always totally non-homologous to zero in $X_G$. Further, all the cases are realizable.\\
\end{theorem}

\begin{theorem} Let $G= \mathbb{S}^1$ act on a finitistic space $X$ of type $(a,0)$ with fixed point set $F$. Suppose $X$ is not totally non-homologous to zero in $X_G$, then either $F= \phi$ or $F\simeq_{\mathbb{Q}} S^r$, where $1 \leq r\leq 3n $ is an odd integer. Moreover, the second possibility is realizable.
\end{theorem}

In studying group actions it is necessary to count the number of conjugacy classes of isotropy subgroups of the group. If $G$ is a compact Lie group and $H$ is a closed subgroup of $G$, then let $[H]$ denote the conjugacy class of $H$ in $G$. The group $G$ is said to act on a space $X$ with finitely many orbit types (FMOT) if the set $\{[G_x]\;|\; x \in X \}$ is finite, where $G_x = \{g \in G\;|\; g.x=x\}$ is the isotropy subgroup at $x$. For a subgroup $H$ of $G$, let $H^0$ denote the component of identity in $H$. Then $G$ is said to act on $X$ with finitely many connective orbit types (FMCOT) if the set $\{[G_x^0]\;|\; x \in X \}$ is finite. The group actions of such type are the correct ones to consider in generalising results about the actions of finite groups to those of compact groups. See for example \cite{A} and \cite{B1} for a detailed account of results concerning such actions. Clearly FMOT implies FMCOT. When working with a field of characteristic zero only FMCOT is needed (\cite{A}, p.131). It is clear that the group $G= \mathbb{S}^1$ acts on any space with FMCOT.

While studying circle actions it is natural to ask what happens for torus actions. We note that:
\begin{itemize}
\item If the $r$-torus $G= {(\mathbb{S}^1)}^r$ acts on a space $X$ of type $(a,0)$ with FMCOT, then by Lemma 4.2.1(1) of Allday and Puppe \cite{A}, there is a subcircle $\mathbb{S}^1 \subset G$ such that their fixed point sets are same, that is $X^{\mathbb{S}^1}=X^G$.
\item If $\mathbb{S}^1$ acts on a space $X$, then we can define $G= {(\mathbb{S}^1)}^r= \mathbb{S}^1 \times {(\mathbb{S}^1)}^{r-1}$ action on $X$ by
$$\big( (g,g_1,...,g_{r-1}), x \big) \mapsto g.x$$
such that $X^{\mathbb{S}^1}=X^G$.
\end{itemize}
Thus, with an additional assumption of FMCOT, our results and examples hold for torus actions also.

\section{Preliminaries}
For details of most of the content in this section we refer to  \cite{A} and \cite{B1}. The cohomology used will be the \v{C}ech cohomology with rational coefficients. We shall use the join $X\star Y$ of two spaces $X$ and $Y$, which is defined as the quotient of $X \times Y \times I$ under the identifications $(x,y_1,0) \sim (x,y_2,0)$ and $(x_1,y,1) \sim (x_2,y,1)$, where $I$ is the unit interval. Thus we are collapsing the subspace $X \times Y \times \{ 0\}$ to $X$ and  $X \times Y \times \{ 1\}$ to $Y$. If $Y$ is a two point space, then $X\star Y$ is called the suspension of $X$ and is denoted by $S(X)$.\\
Observe that,
\begin{itemize}
\item if a group $G$ acts on both $X$ and $Y$ with fixed point sets $F_1$ and $F_2$ respectively, then the induced action on the join $X\star Y$ has the fixed point set $F_1 \star F_2$.
\item a given map $f:X \times Y \to Z$ induces a map $\tilde{f}:X\star Y \to S(Z)$ given by
$$ \tilde{f}\big(\;\overline{(x,y,t)}\; \big)=\overline{(f(x,y),t)}.$$
\end{itemize}

\noindent We say that a map $f:S^{n-1}\times S^{n-1} \to S^{n-1}$ has bidegree $(\alpha,\beta)$, if the map $f|_{S^{n-1} \times \{ p_2\}}$ has degree $\alpha$ and the map $f|_{\{ p_1 \} \times S^{n-1}} $ has degree $\beta$, where $(p_1, p_2) \in S^{n-1}\times S^{n-1}$. The bidegree of $f$ is independent of the choice of $(p_1, p_2)$. The following result is well known.

\begin{proposition}
For every even integer $n\geq 2$ there is a map $\varphi:S^{n-1}\times S^{n-1} \to S^{n-1}$ of bidegree $(2,-1)$.
\end{proposition}
\begin{proof} Define $ \varphi:S^{n-1} \times S^{n-1} \to S^{n-1}$ by $$\varphi (x,y)= y-2(\sum_{i=1}^n x_iy_i)x.$$
If we fix $x=(1,0,...,0)$, then $\varphi(x,y)= (-y_1,y_2,...,y_n)$ which has degree -1. If we fix $y=(1,0,...,0)$, then $ \varphi(x,y)= (1-2{x_1}^2,-2x_1x_2,...,-2x_1x_n)=g(x)$ say. Note that $g$ maps the points $(0,x_2,...,x_n)$ to $(1,0,...,0)$ and is one-to-one for points $(x_1,x_2,...,x_n)$ with $x_1 \neq 0$. Therefore it can be factored into $S^{n-1} \to S^{n-1} \vee S^{n-1} \to S^{n-1}$, where the first map has degree 2 and the second map has degree 1 (see \cite{H} Chapter VII Theorem 10.1). Hence $g$ has degree 2 and $\varphi$ has bidegree $(2,-1)$.\\
\end{proof}

Let $f:S^{n-1}\times S^{n-1} \to S^{n-1}$ be a map of bidegree $(\alpha,\beta)$ and
$$\tilde{f}:S^{2n-1}=S^{n-1}\star S^{n-1} \to S(S^{n-1})=S^n$$
be the induced map. If $Y$ is the complex obtained by attaching a $2n$-cell $e^{2n}$ to $S^n$ via $\tilde{f}$, then for generators $x \in H^n(Y; \mathbb{Z})$ and $y \in H^{2n}(Y; \mathbb{Z})$, we have $x^2=h(\tilde{f})y$ for some integer $h(\tilde{f})$ called the Hopf invariant of $\tilde{f}$. A homotopy of $\tilde{f}$ leaves the homotopy type of $Y$ unchanged and so the Hopf invariant is an invariant of the homotopy class of $\tilde{f}$. The following result relates the Hopf invariant with the bidegree.

\begin{proposition} 
$h(\tilde{f})= \pm \alpha \beta$
\end{proposition}
\noindent See \cite{H}, Chapter XIV, 3.5.\\

\noindent The following facts regarding $\mathbb{S}^1$ actions can be easily deduced.

\begin{proposition}
Let $G= \mathbb{S}^1$ act on a finitistic space $X$ with fixed point set $F$. Suppose that $\sum_{i\geq0} rk H^i(X) < \infty$, then $$ \sum_{i\geq0} rk H^{i}(F)\leq \sum_{i\geq0} rk H^{i}(X).$$
Furthermore, $ \sum_{i\geq0} rk H^{i}(F) = \sum_{i\geq0} rk H^{i}(X)$ if and only if $X$ is totally non-homologous to zero in $X_G$.
\end{proposition}
\noindent See \cite{A}, Theorem 3.10.4.\\

\begin{proposition}
Let $G=\mathbb{S}^1$ act on a finitistic space $X$ with fixed point set $F$. Suppose that $\sum_{i\geq0} rk H^i(X) < \infty$, then $$\chi(X)=\chi(F).$$
\end{proposition}
\noindent See \cite{A}, Corollary 3.1.13 and Remark 3.10.5 (2).\\

We now recall a $\mathbb{S}^1$ action on the unit sphere $S^n$ that we shall use in the following sections. For an odd integer $n$ and $r=(n+1)/2$, the unit sphere $S^n= \{(z_1,z_2,..., z_r) \in \mathbb{C}^r|\sum_{i=1}^{r} |z_i|^2=1 \}$ has a free $\mathbb{S}^1$ action given by
$$\big( z,(z_1,z_2,...,z_r) \big) \mapsto (zz_1,zz_2,...,zz_r).$$
Taking suspension gives a $\mathbb{S}^1$ action on the even dimensional sphere $S^{n+1}$ with exactly two fixed points.\\
\indent We shall also use some elementary notions about vector bundles for which we refer to Husemoller \cite{H}. An action of a group on a vector bundle will be by bundle maps. Theorem 1 is proved in Section 3 and Theorem 2 is proved in Section 4.

\section{Proof of Theorem 1 and examples}

Let $X$ be totally non-homologous to zero in $X_G$. Then by Proposition 3,
$$ \sum_{i\geq 0} rkH^i(F)=\sum_{i\geq 0} rkH^i(X)=4.$$
It is clear that $F$ has at most four components.\vspace{1.5 mm } \\
\indent \textbf{Case (1)} Suppose $F$ has four components, then it is clear that each is acyclic. \\
By Proposition 4, $\chi(X)=\chi(F)=4$ and hence $n$ is even.\vspace{1.5 mm } \\
\indent For $a=0$ and even integer $n$, we can take $X= S^n\vee S^{2n}\vee S^{3n}$. As discussed in the previous section, an even dimensional sphere has a $\mathbb{S}^1$ action with exactly two fixed points. Taking the wedge sum at some fixed points of $S^n, S^{2n}$ and $S^{3n}$ gives a four fixed point $\mathbb{S}^1$ action on $X$.\vspace{1.5 mm } \\
\indent For $a\neq 0$, we know that $X\simeq_{\mathbb{Q}}P^2(n) \vee S^{3n}$. Let $n$ be an even integer, then by  Proposition 1, there is a map $$\varphi:S^{n-1} \times S^{n-1} \to S^{n-1}$$ of bidegree $(2,-1)$. It is clear that $\varphi$ is equivariant with respect to the usual $O(n)$ action on $S^{n-1}$ and the diagonal action on $S^{n-1} \times S^{n-1}$ and hence the induced map $$\tilde{\varphi}:S^{2n-1} \to S^{n}$$ is also equivariant with respect to the induced action. Let $X_n$ denote the mapping cone of $\tilde{\varphi}$ which inherits an $O(n)$ action. By Proposition 2, the Hopf invariant of $\tilde{\varphi}$  is  -2.  Then $X_n \simeq_{\mathbb{Q}} P^2(n)$ and $X_0$ consists of three points. It is clear that if $G \subset O(n)$ acts on $\mathbb{R}^n$ with fixed point set $\mathbb{R}^k$, then the induced action on $X_n$ has fixed point set $X_k$. Let $\mathbb{S}^1 \subset O(n)$ acts on $\mathbb{R}^n$ with exactly one fixed point $\mathbb{R}^0$, then it act on $X_n$ with the fixed point set $X_0$. Let $\mathbb{S}^1$ act on $S^{3n}$ with exactly two fixed points, then the wedge sum $X=X_n \vee S^{3n}$ at fixed points has a $\mathbb{S}^1$ action with exactly four fixed points.\\

\textbf{Case (2)} Suppose that $F$ has three components, then
$$F \simeq_{\mathbb{Q}} S^r \sqcup \{point_1 \} \sqcup \{ point_2 \} ~ \textrm{for some} ~ 1 \leq r \leq 3n .$$
Note that $\chi(F)= 2$ or 4 according as $r$ is odd or even. But $\chi(X)=\chi(F)$ implies that both $n$ and $r$ are even. \vspace{1.5 mm}\\
\indent For $a=0$ and even integer $2\leq r\leq 3n-2$, take $X= S^n\vee S^{2n}\vee S^{3n}$. As $1\leq (3n-r-1)$ is an odd integer, $\mathbb{S}^1$ has a free action on $S^{3n-r-1}$ and the $(r+1)$-fold suspension gives a $\mathbb{S}^1$ action on $S^{3n}$ with $S^r$ as its fixed point set. Taking a two fixed point action on both $S^n$ and $S^{2n}$, the wedge sum at fixed points gives a $\mathbb{S}^1$ action on $X$ with the fixed point set $F \simeq_{\mathbb{Q}} S^r \sqcup \{point_1 \} \sqcup \{ point_2 \}$. \vspace{1.5 mm}\\
\indent For $a\neq 0$, take $X = X_n \vee S^{3n}$. Taking the above three fixed point action of $\mathbb{S}^1$ on $X_n \simeq_{\mathbb{Q}}P^2(n)$ and the action on $S^{3n}$ with $S^r$ as its fixed point set, the wedge sum gives a $\mathbb{S}^1$ action on $X$ with the desired fixed point set.\\

\textbf{Case(3)} Suppose $F$ has two components. Then
$$F\simeq_{\mathbb{Q}} S^r \sqcup S^s, ~ (S^r \vee S^s) \sqcup \{ point \} ~ \textrm{or} ~ P^2(r) \sqcup \{ point \}.$$
If $n$ is odd, $\chi(F)=\chi(X)=0$ and hence
$$F\simeq_{\mathbb{Q}} S^r \sqcup S^s ~ \textrm{or} ~ (S^r \vee S^s) \sqcup \{ point \} ~ \textrm{for odd integers} ~ 1\leq r,s\leq 3n.$$
And if $n$ is even, $\chi(F)= \chi(X)=4$ and hence
$$F\simeq_{\mathbb{Q}} S^r \sqcup S^s ~ \textrm{or} ~ (S^r \vee S^s) \sqcup \{ point \} ~ \textrm{for even integers} ~    2\leq r,s\leq 3n$$
or
$$F\simeq_{\mathbb{Q}} P^2(r)\sqcup \{ point \} ~ \textrm{for some even integer} ~ 2\leq r\leq n .$$
\indent Let $a=0$ and $n$ = 2, 4 or 8. Let $S^{n-1}$ denote the set of complex numbers, quaternions or octanions of norm 1 and $$f:S^{n-1} \times S^{n-1} \to S^{n-1}$$ be the multiplication of complex numbers, quaternions or octanions for $n$ = 2, 4 or 8 respectively. For $n$ = 2 and 4, let $\mathbb{S}^1$ act on $S^{n-1}$ by $$(z,w)\mapsto zw$$
and act on $S^{n-1} \times S^{n-1}$ by $$\big( z,(w_1,w_2) \big) \mapsto (zw_1, w_2).$$
Then $\mathbb{S}^1$ acts freely on both $S^{n-1}$ and $S^{n-1} \times S^{n-1}$ and $f$ is a $\mathbb{S}^1$-equivariant map. Thus the induced map $$\tilde{f}: S^{2n-1} \to S^n$$ is also $\mathbb{S}^1$-equivariant and $\mathbb{S}^1$ act on the mapping cone $M$ of $\tilde{f}$ with $S^n \sqcup \{point \}$ as its fixed point set.\\
\indent For $n =8$ denote an element of $S^{n-1}$ by $(w_1,w_2)$ where $w_1$ and $w_2$ are quaternions. Let $\mathbb{S}^1$ act on $S^{n-1}$ by $$\big( z,(w_1,w_2) \big) \mapsto (zw_1z,w_2)$$
and act on $S^{n-1} \times S^{n-1}$ by $$\Big( z, \big( (w_1,w_2),(w_3,w_4) \big) \Big) \mapsto \big( (zw_1,w_2z),(w_3z,w_4\overline{z}) \big).$$
Then $\mathbb{S}^1$ acts freely on each factor of $S^{n-1} \times S^{n-1}$ and acts on $S^{n-1}$ with fixed point set $S^5$. Note that, $f$ is $\mathbb{S}^1$-equivariant and hence the induced map $\tilde{f}$ is also $\mathbb{S}^1$-equivariant. Thus $\mathbb{S}^1$ acts on the mapping cone $M$ of $\tilde{f}$ and has the fixed point set $S^6 \sqcup \{ point\}$.\\
\indent Note that $f$ has bidegree $(1,1)$ and hence $\tilde{f}$ has Hopf invariant 1. This shows $M\simeq_{\mathbb{Q}} P^2(n)$. Let $\mathbb{S}^1$ act freely on $S^{n-1}$ and act on $S^n$ with $S^r$ as the fixed set for some even integer $r$. Then $\mathbb{S}^1$ acts on $Y=S^{n-1}\star M$ with the fixed point set $S^s \sqcup\{ point\}$, where $s$ is an even integer and hence act on $X=S^n \vee Y\simeq_{\mathbb{Q}}S^n \vee S^{2n} \vee S^{3n}$ with the fixed point set $S^r \sqcup S^s$ or $(S^r \vee S^s) \sqcup \{ point \}$ depending on the wedge sum.\\
For the remaining case, we consider the $\mathbb{S}^1$ action on $X_n \simeq_{\mathbb{Q}}P^2(n)$ with $X_r \simeq_{\mathbb{Q}}P^2(r)$ as its fixed point set, where both $n$ and $r$ are even. Let $\mathbb{S}^1$ act freely on $S^{n-1}$ and act on $S^n$ with exactly two fixed points. Then it acts on $Y=S^{n-1}\star X_n$ with $X_r$ as the fixed point set and hence acts on $X=S^n \vee Y\simeq_{\mathbb{Q}}S^n \vee S^{2n} \vee S^{3n}$ with the fixed point set $F\simeq_{\mathbb{Q}}P^2(r) \sqcup \{ point \}$. Alternatively, one can use Tomter \cite{Tom}, where a $\mathbb{S}^1$ action has been constructed on the actual quaternionic projective space $QP^2$ with the complex projective space $\mathbb{C}P^2$ as its fixed point set and do the above construction.
\vspace{1.5 mm}\\
\indent For $a\neq 0$, take $X=P^2(n)\vee S^{3n}$. As above, we can construct $\mathbb{S}^1$ actions on $X$ with the desired fixed point sets.\\

\textbf{Case(4)} Suppose $F$ has one component. Then either
 $$F\simeq_{\mathbb{Q}} S^r \vee S^s \vee S^t ~ \textrm{for some interers}~ 1 \leq r, s, t \leq 3n$$
or
$$F\simeq_{\mathbb{Q}} S^s \vee P^2(r) ~ \textrm{for some integers} ~ 1 \leq r \leq n ~ \textrm{and} ~ 1 \leq s \leq 3n.$$
\indent As $\chi(F) = \chi(X)$, for $F \simeq_{\mathbb{Q}} S^r \vee S^s \vee S^t$ we must have either $r$, $s$ and $t$ all are even or exactly one of them is even. Similarly, for $F \simeq_{\mathbb{Q}} S^s \vee P^2(r)$ we must have either $s$ and $r$ both even or both odd.\\
\indent For $a=0$, take $X=S^{n}\vee S^{2n} \vee S^{3n}$. If $n$ is even, we take $\mathbb{S}^1$ actions on $S^n$, $S^{2n}$ and $S^{3n}$ having $S^r$, $S^s$ and $S^t$ respectively as fixed point sets, where  $r$, $s$ and $t$ are all even. And if $n$ is odd, we take actions where exactly one of them is even namely $s$. This gives an action on $X$ with $S^{r}\vee S^{s} \vee S^{t}$ as the fixed point set, where the wedge is taken at some fixed points on the subspheres.\\
For the other case take $X=S^n \vee Y$, where $Y=S^{n-1}\star P^2(n)$ and $n$ is even. Consider the $\mathbb{S}^1$ action on $S^n$ with $S^s$ as its fixed point set for some even $s$ and the action on $Y$ fixing $ P^2(r)$ for some even $r$. This can be obtained by taking a free action on $S^{n-1}$ and an action on $P^2(n)$ with $P^2(r)$ as the fixed point set. For example the $\mathbb{S}^1$ action on $\mathbb{R}^n$ with $\mathbb{R}^r$ as its fixed point set ($r$ is even) induces an action on $X_n \simeq_{\mathbb{Q}} P^2(n)$ with $X_r \simeq_{\mathbb{Q}} P^2(r)$ as its fixed point set or one can use \cite{Tom} for $n=4$ and $r=2$. This gives a $\mathbb{S}^1$ action on $X$ with  $F\simeq_{\mathbb{Q}} S^s \vee P^2(r)$.\vspace{1.5 mm }\\
\indent For $a\neq 0$, as above, we can construct an action on $X=P^2(n)\vee S^{3n}$ with $F\simeq_{\mathbb{Q}}S^s \vee P^2(r)$ for even integers $r$ and $s$. In this case the fixed point set cannot be a wedge of three spheres.\vspace{1.5 mm }\\
\indent Now suppose that $n$ is even and $X$ is not totally non-homologous to zero in $X_G$. Then by Proposition 3,
$$ \sum_{i\geq 0} rk H^i(F)\neq \sum_{i\geq 0} rk H^i(X) = 4 $$
and hence
$$\sum_{i\geq 0} rk H^i(F) \leq 3.$$
This gives $\chi(F)=$ -1, 0, 1, 2 or 3. But Proposition 4 gives $\chi(F)=\chi(X)=4$, a contradiction. This completes the proof of the theorem. $\square$

\section{Proof of Theorem 2 and examples}
Let $X$ be not totally non-homologous to zero in $X_G$. Then $n$ is odd and $\chi(F)=\chi(X)=0$. As above $\sum_{i\geq 0} rk H^i(F)\leq 3$ and hence $\chi(F)=$ -1, 0, 1, 2 or 3.\\
But $\chi(F) =0$ and therefore either $F= \phi$ or $F\simeq_{\mathbb{Q}} S^r$, where $1\leq r\leq 3n$ is an odd integer.\vspace{1.5 mm }\\
\indent Recall that when $n$ is odd, we have $a=0$. We now construct a $\mathbb{S}^1$ action on $S^2 \times S^{n+2}$ with $S^3$ as its fixed point set, where $n \geq 3$ is an odd integer \cite[p. 268]{B2}. Let $\eta$ be the Hopf 2-plane bundle over $S^2$ and  $- \eta$ be its inverse, that is $- \eta \oplus \eta$ = trivial 4-plane bundle. Let $\epsilon$ be the trivial $(n-1)$-plane bundle over $S^2$. Then $- \eta \oplus \epsilon$ is a $(n+1)$-plane bundle (where $n+1$ is even) admitting a fibre-wise orthogonal action of $\mathbb{S}^1 \subset O(n+1)$ by bundle maps leaving the zero section fixed (which is the base space $S^2$). Together with the trivial action of $\mathbb{S}^1$ on $\eta $, this gives an action on the trivial $(n+3)$-plane bundle
$$ \eta \oplus (- \eta \oplus \epsilon): ~\mathbb{R}^{n+3} \hookrightarrow S^2 \times \mathbb{R}^{n+3} \longrightarrow S^2$$
whose fixed set is $\eta$. Taking the unit sphere bundles we get an action of $\mathbb{S}^1$ on $S^2 \times S^{n+2}$ with fixed point set $S^3$ (which is the total space of the unit sphere bundle of $\eta$). Now, remove a fixed point from $S^2 \times S^{n+2}$ to obtain a space $Z\simeq_{\mathbb{Q}} S^2 \vee S^{n+2}$ with a $\mathbb{S}^1$ action and  contractible fixed point set. Let  $\mathbb{S}^1$ act trivially on $S^{n-3}$ and consider the induced action on the join $W = S^{n-3}\star Z$ which is homotopically equivalent to $S^n\vee S^{2n}$. This action on $W$ has a contractible fixed point set. For a given odd integer $1\leq r\leq 3n$, consider the $\mathbb{S}^1$ action on $S^{3n}$ with $S^r$ as its fixed point set. Then the wedge of $W$ and $S^{3n}$ at a fixed point is a space $X\simeq_{\mathbb{Q}} S^n\vee S^{2n}\vee S^{3n}$ and has a $\mathbb{S}^1$ action with its fixed point set $F \simeq_{\mathbb{Q}} S^r$.  $\square$ \\

\noindent \textbf{Remark.} It is clear that there is no free $\mathbb{S}^1$ action on $S^n\vee S^{2n}\vee S^{3n}$. The author does not know of any free $\mathbb{S}^1$ action on a space $X\simeq_{\mathbb{Q}} S^n\vee S^{2n}\vee S^{3n}$. In case it exists, it will be interesting to determine the cohomology ring of the orbit space, for which one can exploit the Leray spectral sequence associated to the Borel fibration $X\hookrightarrow X_{\mathbb{S}^1} \longrightarrow B_{\mathbb{S}^1}$.\\

\noindent \textbf{Acknowledgement.} The author is grateful to the referee for several useful remarks which improved the presentation and exposition of the paper. In particular, the reference Tomter \cite{Tom} was suggested by the referee, which is used in examples in case (3) and (4) of section 3.

\bibliographystyle{amsplain}

\end{document}